\documentclass[10pt,leqno, english]{amsart}
\title[Induced character in equivariant K-theory and wreath products]
{Induced character in equivariant K-theory, wreath products and pullback of groups}
\author{German Combariza}
\address{Departamento de Matem\'aticas. \\Pontificia Universidad Javeriana\\Cra. 7 No. 43-82 - Edificio Carlos Ort\'iz 5to piso\\ Bogot\'a D.C, Colombia}

\email{germancombariza@javeriana.edu.co}
\urladdr{https://sites.google.com/site/combariza/research}

\author{Juan Rodriguez}
\address{UMPA L'unit\'e de Math\'ematiques Pures et Appliqu\'ees \\ENS de Lyon site Monod UMPA UMR 5669 CNRS 46, all\'e d'Italie\\69364 Lyon Cedex 07, France }

\email{juan-esteban.rodriguez-camargo@ens-lyon.fr}

\author{Mario Vel\'asquez}
\address{Departamento de Matem\'aticas. \\Pontificia Universidad Javeriana\\Cra. 7 No. 43-82 - Edificio Carlos Ort\'iz 5to piso\\ Bogot\'a D.C, Colombia}

\email{mario.velasquez@javeriana.edu.co}
 \urladdr{https://sites.google.com/site/mavelasquezm/}
         \date{\today}
         
\usepackage{amssymb}
\usepackage{amsmath}
\usepackage{hyperref}
\usepackage{verbatim}
\usepackage{pdfsync}
\usepackage{calc}
\usepackage{enumerate,amssymb}
\usepackage[arrow,curve,matrix,tips,2cell]{xy}
\SelectTips{eu}{10} \UseTips
\UseAllTwocells
\usepackage{graphicx}
\usepackage{pb-diagram}
\usepackage{amsmath}
\usepackage{hyperref}
\usepackage{verbatim}
\usepackage{pdfsync}
\usepackage{mathtools}
\usepackage{calc}
\usepackage{enumerate,amssymb}
\usepackage[arrow,curve,matrix,tips,2cell]{xy}
\SelectTips{eu}{10} \UseTips
\UseAllTwocells
\usepackage{graphicx}
\usepackage{pb-diagram}
\usepackage{amscd}
\usepackage[mathscr]{euscript}

\DeclareMathAlphabet\EuR{U}{eur}{m}{n}
\SetMathAlphabet\EuR{bold}{U}{eur}{b}{n}

\makeindex             

\usepackage{pstricks}
\usepackage{pst-3dplot}

\theoremstyle{plain}
\newtheorem{theorem}{Theorem}[section]
\newtheorem{lemma}[theorem]{Lemma}
\newtheorem{proposition}[theorem]{Proposition}
\newtheorem{corollary}[theorem]{Corollary}

\theoremstyle{definition}
\newtheorem{definition}[theorem]{Definition}
\newtheorem{example}[theorem]{Example}

\newtheorem{remark}[theorem]{Remark}

{\catcode`@=11\global\let\c@equation=\c@theorem}

\def\spinc{\text{Spin}^c}

\newcommand{\xycomsquare}[8]                   
{\xymatrix
	{#1 \ar[r]^{#2} \ar[d]^{#4} &
		#3 \ar[d]^{#5}  \\
		#6\ar[r]^{#7} &
		#8
	}
}

\newcommand{\calf}{\mathcal{F}}

\def\c{\mathfrak{C}}
\def\s{\mathcal{S}}

\newcommand{\IC}{{\mathbb C}}

\newcommand{\IZ}{{\mathbb Z}}

\newcommand{\frakS}{{\mathfrak S}}

\newcommand{\charac}{\operatorname{char}}

\newcommand{\Class}{\operatorname{Class}}

\newcommand{\Hom}{\operatorname{Hom}}

\newcommand{\im}{\operatorname{im}}

\newcommand{\Ind}{\operatorname{Ind}}

\newcommand{\res}{\operatorname{res}}

\newcommand{\pt}{\{\bullet\}}

\newcommand{\higherlim}[3]{{\setbox1=\hbox{\rm lim}
		\setbox2=\hbox to \wd1{\leftarrowfill} \ht2=0pt \dp2=-1pt
		\mathop{\vtop{\baselineskip=5pt\box1\box2}}
		_{#1}}^{#2}#3}

\newcommand{\version}[1]                       
{\begin{center} last edited on #1\\
		last compiled on \today \\
		name of texfile: \jobname
	\end{center}
}

\newcounter{commentcounter}

\makeindex
\begin{document}
	
	\begin{abstract}
		Let $G$ be a finite group, $X$ be a compact $G$-space. In this note we study the $(\IZ_+\times\IZ/2\IZ)$-graded algebra
		$$\calf^q_G(X)=\bigoplus_{n\geq0}q^n\cdot K_{G\wr\mathfrak{S}_n}(X^n)\otimes\IC,$$defined in terms of equivariant K-theory with respect to wreath products as a symmetric algebra. More specifically, let $H$ be another finite group and $Y$ be a compact $H$-space, we give a decomposition of $\calf^q_{G\times H}(X\times Y)$ in terms of $\calf^q_G(X)$ and $\calf^q_H(Y)$. For this, we need to study the representation theory of pullbacks of groups. We discuss also some applications of the above result to equivariant connective K-homology.

	\end{abstract}	
	\maketitle
\section*{Notation}
In this note we denote by $\frakS_{n}$ the symmetric group in $n$ letters. Let $G$ be a finite group, let $g,g'\in G$, we say that $g$ and $g'$ are conjugated in $G$ (denoted by $g\sim_Gg'$) if there is $s\in G$ such that $g=sg's^{-1}$. We denote by $$[g]_G=\{g'\in G\mid g\sim_Gg'\}$$to the conjugacy class of $g$ in $G$ (or simply by $[g]$ when $G$ is clear from the context). We denote by $G_*$ the set of conjugacy classes of $G$. We denote  by $C_G(g)$ the centralizer of $g$ in $G$. Also  $R(G)$ will be the complex representation ring of $G$, with operations given by direct sum and tensor product, and generated as abelian group by the isomorphism classes of irreducible representations of $G$. The \emph{class function ring of $G$} is the set  $$\Class(G)=\{f:G\to\IC\mid f \text{ is constant in conjugacy classes}\}$$ with the usual operations.

\section{Introduction}

Let $X$ be a finite CW-complex. In \cite{segal1996} Segal studied the vector spaces$$\calf(X)=\bigoplus_{n\geq0}K_{\frakS_{n}}(X^n)\otimes\IC,$$
these spaces carries several interesting structures, for example they admit a Hopf algebra structure with the product defined using induction on vector bundles and the coproduct defined using restriction.

Later in \cite{wang2000}, Wang generalizes Segal's work to an equivariant context. Let $G$ be a finite group and $X$ be a finite  $G$-CW-complex, Wang defines the vector space
$$\calf_G(X)=\bigoplus_{n\geq0}K_{G\wr\frakS_{n}}(X^n)\otimes\IC,$$
where $G\wr\frakS_{n}$ denotes the wreath product acting naturally over $X^n$. Wang proves that $\calf_G(X)$ admits similar structures as $\calf(X)$. In particular $\calf_G(X)$ has a description as a supersymmetric algebra in terms of $K_G(X)\otimes\IC$. In this paper we present a proof of the above fact using a explicit description of the character of an induced vector bundle. This description indicates that $\calf_G(X)$ has the size of the Fock space of a Heisenberg superalgebra.

Following ideas of \cite{segal1996}, in \cite{ve2015} appears another reason to study $\calf_G(X)$. When $X$ is a $G$-spin$^c$-manifold of even dimension,  $\calf_G(X)$ is isomorphic to the homology with complex coefficients of the $G$-fixed point set of a based configuration space $\c(X,x_0,G)$ whose $G$-equivariant homotopy groups corresponds to the reduced $G$-equivariant connective K-homology groups of $X$. This description allows to relate generators of $\calf_G(X)$ with some homological versions of the Chern classes.

Let $G$ and $H$ be finite groups, $X$ be a finite $G$-CW-complex and $Y$ be a finite $H$-CW-complex, we also prove a K\"unneth formula for $\calf_{G\times H}(X\times Y)$, obtaining an isomorphism
$$\calf_{G\times H}(X\times Y)\cong\calf_G(X)\otimes_{\mathcal{F}(\pt)}\calf_H(Y)$$that is compatible with the decomposition as a supersymmetric algebra. In order to do this, we need to study the representation theory of pullbacks of groups.

Let \begin{equation*}
\xymatrix{
	\Gamma \ar[r]^{p_2}\ar[d]_{p_1} & G\ar[d]_{\pi_2}\\
	H \ar[r]^{\pi_1} &  K}\end{equation*}
be a pullback diagram of finite groups, with $\pi_1$ and $\pi_2$ surjective, in this case $\Gamma$ can be realized as a subgroup of $G\times H$. We prove that when $\Gamma$ is conjugacy-closed (see Definition \ref{conjugacy-closed}) in $G\times H$ then we  have a ring isomorphism
$$\Class(\Gamma)\cong \Class(H)\otimes_{\Class(K)}\Class(G).$$

Moreover we use that some conjugacy classes in $(G\times H)\wr\frakS_{n}$ are closed in $(G\wr\frakS_n)\times (H\wr\frakS_{n})$ to prove the K\"unneth formula for the algebra $\calf_{G\times H}(X\times Y)$.

This paper is organized as follows:

In Section 2 we recall basic facts about equivariant K-theory, in particular we recall the construction for the character. Following ideas of \cite{serre} we give an explicit definition of the induced bundle and recall a formula (proved in \cite[Thm. D]{kuhninduction}) for a character of the induced bundle. In Section 3 we recall basic facts about wreath products and its action over $X^n$. In Section 4 we recall the definition of $\calf_G(X)$ and give another way to obtain the description as a supersymmetric algebra using the formula of the induced character. In Section 5 we study the representation theory of pullbacks. In Section 6 we recall some basic properties of semidrect products of direct products. In Section 7 we use results in Section 5 to give a K\"unneth formula for the Hopf algebra $\calf_{G\times H}(X\times Y)$. In Section 8 we do some final remarks about the relation of $\calf(X)$ and homological versions of Chern classes.

\section{Induced character in equivariant K-theory}
In this section we recall a decomposition theorem for equivariant K-theory with complex coefficients obtained by Atiyah and Segal in \cite{a-s1989}. In the next section we use that result to give a simple description of $\calf^q_G(X)$. In this paper all CW-complexes (and $G$-CW-complexes) that we consider are finite.
\begin{definition}
	Let $X$ be a $G$-space. A $G$-vector bundle over $X$ is a map $p:E\to X$, where $E$ is a $G$-space satisfying the following conditions
	\begin{enumerate}
		\item $p:E\to X$ is a vector bundle.
		\item $p$ is a $G$-map
		\item For every $g\in G$ the left translation $E\to E$ by $g$ is bundle map.
	\end{enumerate}
\end{definition}
If $p:E\to X$ is a $G$-vector bundle we define \emph{the fiber over $x\in X$} to the set$$p^{-1}(x)=\{v\in E\mid p(v)=x\},$$when $p$ is clear from the context we also denote this set by $E_x$. Also if $H\subseteq G$ is a subgroup, we can consider $E$ as a $H$-vector bundle over $X$, we denote it by $\res_H^G(E)$.

\begin{definition}Let $X$ and $Y$ be $G$-spaces. If $p:E\to Y$ is a $G$-vector bundle and  $f:X\to Y$ is a $G$-map, then the pullback $p^*E \to X$ is a $G$-vector bundle over $X$ defined as $$p^*E=\{(e,x)\in E\times X\mid p(e)=f(x)\}.$$ When $i:X\to Y$ is an inclusion we usually denote $i^*(E)$ by $E\mid Y$.
\end{definition}

Details about $G$-vector bundles can be found in \cite{at89}. 

\begin{definition}
	Let $G$ be a group, let $X$ be a finite $G$-CW-complex (see \cite{t87}), the \emph{equivariant K-theory group of $X$},denoted by $K_G(X)$ is defined as the Grothendieck group of the monoid of isomorphism classes of $G$-equivariant vector bundles over $X$ with the operation of direct sum.
	The functor $K_G(-)$ could be extended to an equivariant cohomology theory $K^*(-)$, defining for $n>0$: $$K_G^{-n}(X)=\ker\left(K_G(X\times S^n)\xrightarrow{i^*}K_G(X)\right).$$And  for any $G$-CW-pair $(X,A)$, set$$K_G^{-n}(X,A)=\ker\left(K_G^{-n}(X\cup_AX )\xrightarrow{i_2^*}K_G^{-n}(X)\right).$$Finally for $n<0$
	$$K_G^{-n}(X)=K^n_G(X)\text{ and }K_G^{-n}(X,A)=K^n_G(X,A).$$
	For more details about equivariant K-theory the reader can consult \cite{segal1968}.
\end{definition}
\begin{example}
 If the action of $G$ over $X$ is free, then there is a canonical isomorphism of abelian groups $$K_G(X) \cong K(X/G).$$
 \end{example}
\begin{example}
 If the action of $G$ over $X$ is trivial, then there is a canonical isomorphism of abelian groups  $$K_G(X)\cong R(G)\otimes_\IZ K(X),$$when $R(G)$ denotes the (complex) representation ring of $G$. In particular when $X=\pt$ we obtain $$K_G(\pt)\cong R(G).$$
\end{example}


If $Y$ is a finite $G$-CW-complex, we can define a $G$-action on $K(Y)$. Let $g\in G$, the pullback $$g^*:K(Y)\to K(Y),$$defines a $G$-action over $K(Y)$. We will need the following lemma.
\begin{lemma}
	Let $Y$ be a finite $G$-CW-complex, then $$K(Y/G)\otimes\IC\cong K(Y)^G\otimes\IC$$
\end{lemma}
\begin{proof}It is a consequence of the Chern character and the analogous fact for singular cohomology.
%
\end{proof}
In \cite{a-s1989}  a character for equivariant K-theory is constructed, that generalizes the character of representations. We will recall this construction briefly. Let $E$ be a $G$-vector bundle over $X$ and $g\in G$. Note that $X^g$ is a $C_G(g)$-space, then if $E$ is a $G$-vector bundle, $E|X^g$ is canonically a $C_G(g)$-vector bundle over $X^g$. Considering the action given by pullback we have that the isomorphism class $[(E|X^g)]\in K(X^g)$ is a $C_G(g)$-fixed point. Then $[(E|X^g)]\in K(X^g)^{C_G(g)}$. Finally for every element $\lambda\in S^1$, we can form the vector bundle of $\lambda$-eigenvectors considering the action of the element $g$ over $\pi(E|X^g)$ denoted by $\pi(E|X^g)_\lambda$. Then we can define a map

\begin{eqnarray*}
	\charac_G:  K_G(X)\otimes\IC &\to& \bigoplus_{[g]}K(X^g)^{C_G(g))}\otimes\IC\\
\left[E\right]&\mapsto&  \left( \bigoplus_{\lambda\in S^1}[\pi(E|_{
	X^g})_\lambda]\otimes\lambda \right)_{[g]}.
\end{eqnarray*}
Using the above Lemma we identify $K(X^g)^{C_G(g)}$ with $K(X^g/{C_G(g)})$.
\begin{theorem}
	The map $\charac_G$ is an isomorphism of complex vector spaces.
\end{theorem}
For a proof of the theorem see \cite{a-s1989}. 


\subsection{The induced bundle}

Now we will give an explicit construction of the induced vector bundle. It is a direct generalization of the induced representation defined for example in Section 3.3 in \cite{serre}.




Let $H\subseteq G$ be a subgroup of $G$ and $E\xrightarrow{\pi} X$ an $H$-vector bundle over a $G$-space $X$. If we choose an element from each left coset of $H$, we obtain a subset $R$ of $G$ called a \textit{system of representatives of $H\diagdown G$}; each $g\in G$ can be written uniquely as $g=sr$, with $r\in R=\{r_1,\ldots,r_n\}$ and $s\in H$, $G=\coprod_{i=1}^nHr_i$, we suppose that $r_1=e$ the identity of the group $G$. Consider the vector bundle $F=\bigoplus_{i=1}^n(r_i)^*E$, with projection $\pi_F:F\to X$ and consider the following $G$-action defined over $F$:

Let $f\in F$, then $$f=f_{r_1}\oplus\cdots\oplus f_{r_n},$$ where $f_{r_i}\in(r_i)^*E$. If $\pi_F(f)=x$ then $f_{r_i}=(x,e)$, where $e\in E_{r_ix}$.
Let  $g\in G$, note that $r_ig^{-1}$ is in the same left coset of some $r_j$, i.e. $r_ig^{-1}=sr_j$, for some $s\in H$. Define $$g(f_{r_i})=(gx,s^{-1}e)\in (r_j)^*E,$$ and define the action of $g$ on $f$ by linearly.

 Now we will see that $F$ does not depend on the set of representatives up to isomorphism. Let $\{r_1',\ldots,r_n'\}$ be another set of representatives of $H\diagdown G$ and let $F'=\bigoplus_{i=1}^n(r_i')^*E$. By reordering  we can assume that $r_i$ and $r_i'$ are in the same left coclass, then $r_i'r_i^{-1}\in H$.

 We have an isomorphism of vector bundles over $X$ \begin{align*}
 r_i'r_i^{-1}:(r_i)^*E&\rightarrow(r_i')^*E\\
 (x,e)&\to (x,r_i'r_i^{-1}e)
 \end{align*}inducing an isomorphism of $G$-vector bundles  $$r_1'r_1^{-1}\oplus\ldots\oplus r_n'r_n^{-1}:F\to F'.$$  We only need to verify that this map commutes with the action of $G$. To see this, let $g\in G$ and $f_{r_i}=(x,e)\in (r_i)^*E$, there exist $s,s'\in H$ such that 
 \begin{equation}\label{ind1}
 r_ig^{-1}=sr_j \text{ and }r_i'g^{-1}=s'r_j'.
 \end{equation} Note that $gf_{r_i}\in (r_j)^*E$, then 
 $$(r_j'r_j^{-1})g(f_{r_i})=(gx,r_j'r_j^{-1}s^{-1}e).$$
 
 On the other hand $$g(r_i'r_i^{-1}f_{r_i})=(gx,(s')^{-1}(r_i'r_i^{-1}e)),$$but we know from (\ref{ind1})
 $$(s')^{-1}r_i'r_i^{-1}=r_j'r_j^{-1}s^{-1}.$$
 Then the map $r_1'r_1^{-1}\oplus\ldots\oplus r_n'r_n^{-1}$ commutes with the $G$-action and then $F$ and $F'$ are isomorphic as $G$-vector bundles.
We will denote the $G$-vector bundle $F$ defined above by $\Ind_H^G(E)$. Summarizing we have.

\begin{theorem}
	Let $G$ be a finite group, let $H\subseteq G$ be a subgroup. Let $X$ be a $G$-CW-complex, and let $E$ be a $H$-vector bundle over $X$, there is a unique  $G$-vector bundle $\Ind_H^G(E)$ over $X$, up to isomorphism of $G$-vector bundles such that for every $G$-vector bundle $F$ over $X$ we have a natural identification$$\Hom_G(\Ind_H^G(E),F)\cong\Hom_H(E,\res_H^G(F)).$$
\end{theorem}
\begin{proof}Only remains to prove the identification. Let $\xi\in\Hom_G(\Ind_H^G(E),F)$, recall that we have an inclusion of $H$-vector bundles 
	\begin{align*}E&\to\Ind_H^G(E)\\
	v\in E_x&\mapsto(x,v).\end{align*}

Define $r(\xi)\in\Hom_H(E,\res_H^G(F))$ as follows, if $v\in E_x$
$$r(\xi)(v)=\xi(x,v).$$

It is clear that $r(\xi)\in\Hom_H(E,\res_H^G(F))$. On the other hand if $\eta\in \Hom_H(E,\res_H^G(F))$, define $I(\eta):\Ind_H^G(E)\to F$ as follows, if $v_i\in r_i^*E$, then $v_i=(x,v)$ with $x\in X$ and $v\in E_{r_ix}$, then we define $$I(\eta)(v_i)=r_i^{-1}(\eta(v)).$$Extending linearly $I(\eta)$ to $\Ind_H^G(E)$.

Now we will see that $I(\eta)$ is $G$-equivariant. Let $g\in G$, let $s\in H$ such that $$r_ig^{-1}=sr_j,$$then $$g\cdot v_i=(gx,s^{-1}v).$$Now, 
\begin{align*}
I(\eta)(g\cdot v_i)&=I(\eta)(gx,s^{-1}v)\\
&=r_j^{-1}(\eta(s^{-1}v))\\&=r_j^{-1}s^{-1}\eta(v)\\&=gr_i^{-1}\eta(v)\\&=gI(\eta)(v_i).
\end{align*}Then $I(\eta)\in\Hom_G(\Ind_H^G(E),F)$. Now we will see that $r$ and $I$ are inverse of each other.
It is clear that  $r(I(\eta))=\eta$. On the other hand, 
\begin{align*}
I(r(\xi))(v_i)&=r_i^{-1}(r(\xi)(v))\\&=r_i^{-1}\xi(r_ix,v)\\&=\xi(x,v)\\&=\xi(v_i).
\end{align*}
\end{proof}

%
We have a formula for the character of an induced $H$-vector bundle, it is a particular case of  a formula for induced character of generalized cohomology theories in \cite{kuhninduction} and \cite{kuhnchar}. We include a proof for completeness.

\begin{theorem}[Formula for the induced character]\label{Main Theorem}
	Let $X$ be a $G$-CW-complex, let $H$ be a subgroup of $G$, let $h$ be the order of $H$ and $E$ be a $H$-vector bundle, consider the map
	$$\charac_G\circ\Ind_H^G:K_H(X)\otimes\IC\to\bigoplus_{[g]}K(X^g)^{C_G(g)}\otimes\IC,$$
	let $R$ be a system of representatives of $H\diagdown G$. For each $g\in G$, we have
	\begin{align*}\charac_G(g)\circ\Ind_H^G([E])&=\bigoplus_{r\in R,r^{-1} gr\in H}r^*\left(\charac^H(r^{-1}gr)([E])\right)\\&=\frac{1}{h}\bigoplus_{r\in G,r^{-1} gr\in H}r^*\left(\charac^H(r^{-1}gr)([E])\right).\end{align*}
\end{theorem}
\begin{proof}Our explicit definition of the induced bundle allows us to proof this result just by adapting the proof for representations contained in \cite{serre}. The vector bundle $F=\Ind_H^G(E)$ is the direct sum $\bigoplus_{i=1}^nr_i^*E$, with $R=\{r_1,\ldots,r_n\}$. $$H\diagdown G=\{Hr_1,\ldots,Hr_n\}.$$ 
	We know from the definition of the induced bundle that if we write $r_ig^{-1}$ in the form $sr_j$ with $r_j\in R$ and $s\in H$, then $g$ sends $r_i^*E$ to $r_j^*E$. Considering the action of $g$ in $H\diagdown G$ , we have that\[\charac_G(g)\left(\Ind_H^G(E)\right)=\charac_G(g)\left(\bigoplus_{Hr_i=Hr_ig^{-1}}r_i^*E\oplus\bigoplus_{Hr_i\neq Hr_ig^{-1}}r_i^*E\right)\]
	Note that $g$ acts in each term of the direct sum on the right hand side, hence the right hand side of the above equation can be written as
	\[\charac_G(g)\left(\bigoplus_{Hr_i=Hr_ig^{-1}}r_i^*E\right)\oplus\charac_G(g)\left(\bigoplus_{Hr_i\neq Hr_ig^{-1}}r_i^*E\right)\]
	
	We will see that $\charac_G(g)\left(\bigoplus_{Hr_i\neq Hr_ig^{-1}}r_i^*E\right)=0$. Because each 0-dimensional bundle is trivial, it suffices to check that this condition holds on fibers, i.e. 
	
	$$\charac_G(g)\left(\bigoplus_{Hr_i\neq Hr_ig^{-1}}r_i^*E\right)_x=0,$$ for all $x\in X^g$. If we fix a basis for $(r_i^*E)_x$, the trace of the matrix representing the action of $g$ is zero because $Hr_i\neq Hr_ig^{-1}$.\\
	
	Now, if $Hr_i=Hr_ig^{-1}$ we have that $r_igr_i^{-1}=s_i$ with $s_i\in H$. Thus as the character is invariant under conjugation $$\charac_G(g)(r_i^*E)=\charac_H(s_i)(r_i^*E).$$ Finally as the character commutes with pullbacks we have that,
	\begin{align*}\charac_G(g)(\Ind_H^G(E))&=\bigoplus_{r\in R, rgr^{-1}\in H}r_i^*\left(\charac_H(rgr^{-1})(E)\right)\\
	&=\frac{1}{h}\bigoplus_{s\in G, sgs^{-1}\in H}r_i^*\left(\charac_H(s^{-1}hs)(E)\right).\end{align*}	\end{proof}

\section{Wreath product and its action on $X^n$}
Let $C$ be a set. There is a natural action of $\frakS_n$ on $C^n$ defined as
\[
\sigma\bullet (c_1,\ldots,c_n)=(c_{\sigma^{-1}(1)},\ldots, c_{\sigma^{-1}(n)}) 
\]
if $G$ is a group, we define the wreath product as the semidirect product 
\[
G_n=G\wr \frakS_n=G^n\rtimes \frakS_n.
\]
This section is dedicated to describe the conjugacy classes and centralizers of elements in $G_n$, we follow \cite[Chapter 1, App. B]{symmetric} and \cite{wang2000}.


First, we must recall the conjugacy classes in $\frakS_n$. Two elements $s_1, s_2 \in \frakS_n$ are conjugate if their cycle factorization correspond to the same partition of $n$. For example the elements $(1,2)(3,4,5)$ and $(1,4)(2,3,5)$ are conjugated and correspond to the partition $5 = 2+3$. Note that every partition of $n$ can be view as a function $m:\{1,2,\cdots, n\}\to\mathbb{N}$ as follows. If $s \in \frakS_n$ then $m_s(r)$ is the number of $r$-cyces in $s$. Now in the general case, if $x=(\bar{g},s)\in G_n$, then $s$ can be decomposed as a product of disjoint cycles, if $z=(i_1i_2\ldots i_r)$  is one of these cycles, the element $g_{i_r}g_{i_{r-1}}\cdots g_{i_1}$ is called the \emph{cycle product} of $x$ corresponding to $z$. 



Recall that  $G_*$ denotes the set of conjugacy classes of $G$. If $x=(\bar{g},s)\in G_n$, let $\rho(x) = m_x(r,c)$ denote the number of $r$-cycles in $s$ whose cycle product belongs to $c$, where $c\in G_*$ and $r\in\{1,2,\cdots n\}$.  In this way every element $x\in G_n$ determines a matrix $\rho(x)=m_x(r,c)$ of non-negative  integers such that $\sum_{r,c}rm_x(r,c)=n$.


For example let $G$ be the cyclic group, $\{g^0,g^1, g^2, g^3\}$, of $4$ elements generated by $g$ and $s=(1,2)(3,4,5) \in\mathfrak{S}_5$. If $x=(g,g,g,g,g,s)$ then $\rho(x) = m_x(r,c)$ looks like
\begin{table}[h]
\centering
\begin{tabular}{ccccc}
                           & $g^0$ & $g^1$ & $g^2$ & $g^3$                  \\
\multicolumn{1}{c|}{$r=1$} & 0     & 0     & 0     & \multicolumn{1}{c|}{0} \\
\multicolumn{1}{c|}{$r=2$} & 0     & 0     & 1     & \multicolumn{1}{c|}{0} \\
\multicolumn{1}{c|}{$r=3$} & 0     & 0     & 0     & \multicolumn{1}{c|}{1} \\
\multicolumn{1}{c|}{$r=4$} & 0     & 0     & 0     & \multicolumn{1}{c|}{0} \\
\multicolumn{1}{c|}{$r=5$} & 0     & 0     & 0     & \multicolumn{1}{c|}{0}
\end{tabular}
\end{table}



the map $\rho:G_n\to M_{n,v}(\mathbb{Z})$ is called the \emph{type} of $x\in G_n$, where $v=|G|$. 

\begin{proposition}\label{conjugacyclasses}
	Two elements in $G_n$ are conjugate iff they have the same type.
\end{proposition}
\begin{proof}
	See \cite[Appendix 1.B]{symmetric}
	\end{proof}

By the above proposition we can assume that every element $x\in G_n$ is conjugated to a product of  elements 
of the form
$$((g,1,\ldots,1),(i_{u_1},\ldots,i_{u_r})).$$
Denote by $g_r(c)=((g,1,\ldots,1),(1,\ldots,r)).$ 
\begin{proposition}\label{centralizer}
The elements in the centralizer $C_{G_n}(g_n(c))$ are of the form $((gz,\ldots,\underbrace{z}_{k+1},\ldots,gz),(1,\ldots,n)^k)$, with $z\in C_G(g)$. Moreover   $C_{G_n}(g_n(c))\cong C_G(g)\times\langle (1,\ldots n)\rangle$.
\end{proposition}
\begin{proof}
	It follows from a direct computation.\end{proof}

We have described centralizers of elements in $G_n$ and in the next section we will use this description  to write $\charac_{G_n}$ in terms of $\charac_G$.

Let $X$ be a $G$-space, there is canonical $G_n$-action over $X^n$ defined from the $G$-action over $X$
\begin{align*}
G_n\times X^n&\to X^n\\
((\bar{g},\sigma),\bar{x})&\mapsto \bar{g}(\sigma\bullet\bar{x})
\end{align*}
where $\bar{g}$ acts component-wise.

In order to relate  $\charac_{G_n}$with with $\charac_G$ we need to describe the fixed point set of a representative of each conjugacy class of $G_n$. Let us start with the conjugacy classes of elements $(\bar{g},\sigma)$ where $\sigma$ is an $n$-cycle. To this end we will need the following result.
\begin{proposition}\label{cycle}Let $\zeta=((g,1,\ldots,1),\sigma)$ whith $g\in G$ and $\sigma$ is a $n$-cycle. There is a canonical homeomorphism

\[
\left(X^n\right)^{\zeta}/C_{G_n}(\zeta)\cong X^{g}/C_G(g).
\]

\end{proposition}
\begin{proof}We can assume $\sigma=(1,\ldots,n)$.
Let $(x_1,\ldots,x_n)\in (X^n)^\zeta$, then 
\[
\zeta(x_1,\ldots,x_n)=(x_1,\ldots,x_n)
\]it implies
\[
(gx_n,x_1,\ldots,x_{n-1})=(x_1,\ldots,x_n).
\]
Therefore
\[
x_n=x_{n-1}=\cdots=x_1,\;\;\; gx_n=x_1,
\]
and then $(x_1,\ldots, x_n)=(y,\ldots, y)$ lies in the diagonal and $y\in X^g$. This proves that $(X^n)^{\zeta}\cong X^{g}$.  On the other hand, if $\bar{b}\in C_{G_n}(\zeta)$ then by Proposition \ref{centralizer} $$\bar{b}=(((gz,\ldots,\underbrace{z}_{k+1},\ldots,gz),\sigma^{k})$$ where $z\in C_G(g)$. Then we obtain
\[
\bar{b}(y,\ldots,y)=(((gz,\ldots,\underbrace{z}_{k+1},\ldots,gz),\sigma^{k})\cdot(y\ldots,y)=(gzy,\ldots,\underbrace{zy}_{k+1},\ldots,gzy),
\]
showing that the orbit of $(y,\ldots,y)$ by $C_{G_n}(\zeta)$ is  $$\{(gzy,\ldots,\underbrace{zy}_{k+1},\ldots,gzy)\,:\, z\in \,C_G(g)\}.$$ This proves the result.\end{proof}


\section{Fock space}

Let $X$ be a $G$-space, by the equivariant Bott periodicity theorem we know that $K^*_G(X)=K^0_G(X)\oplus K_G^1(X)$ is a $\IZ/2\IZ$-graded group. Denote by \begin{align*}
\calf_G(X)=\bigoplus_{n\geq0}K_{G_n}^*(X^n)\otimes\IC,&& \calf^q_G(X)=\bigoplus_{n\geq0}q^nK_{G_n}^*(X^n)\otimes\IC
\end{align*}
where $q$ is formal variable giving a $\IZ_+$-grading in $\calf_G^q(X)$. They both have a natural structure of abelian groups, we endow them with a product $\cdot$, defined as the composition of the induced bundle and the K\"unneth isomorphism $\boxtimes$ (see \cite{kunnethk-theory})

$$K_{G_n}^*(X^n)\times K_{G_m}^*(X^m)\xrightarrow{\boxtimes}K_{G_n\times G_m}^*(X^{n+m})\xrightarrow{\Ind}K_{G_{n+m}}^*(X^{n+m})$$

\begin{proposition}
	With the above operations $\calf_G^q(X)$ is a commutative $(\IZ_+\times\IZ/2\IZ)$-graded ring.
\end{proposition}
\begin{proof}
The associativity follows from the following fact. Let $[E_1]\in K_{G_n}^*(X^n)$, $[E_2]\in K_{G_m}^*(X^m)$ and $[E_3]\in K_{G_k}^*(X^k)$, then\begin{align*}(E_1\cdot E_2)\cdot E_3&\cong\Ind_{G_{n+m}\times G_k}^{G_{n+m+k}}\left(\Ind_{G_n\times G_n}^{G_{n+m}}\left(E_1\boxtimes E_2\right)\boxtimes E_3\right)\\&\cong\Ind_{G_{n}\times G_m\times G_k}^{G_{n+m+k}}\left(E_1\boxtimes E_2\boxtimes E_3\right)\\&\cong \Ind_{G_{n}\times G_{m+k}}^{G_{n+m+k}}\left(E_1\boxtimes\Ind_{G_m\times G_k}^{G_{m+k}}\left(E_2\boxtimes E_3\right)\right).\end{align*}
For the graded commutativity, let $[E_1]\in K_{G_n}^*(X^n)$ and $[E_2]\in K_{G_m}^*(X^m)$, we will prove that $E_1\cdot E_2$ and $E_2\cdot E_1$ has the same character as $G_{n+m}$-vector bundles over $X^{n+m}$.

Consider two inclusions of $\frakS_{n}$ into $\frakS_{n+m}$. The first one is the inclusion by the first $n$ letters denoted by $S_n\xrightarrow{i_1^n}S_{n+m}$, the second one is the inclusion by the last $n$ letters denoted by $S_n\xrightarrow{i_2^m}S_{n+m}$. Let $x=(\bar{g},\sigma)\in G_{n+m}$ and let $r=(\bar{h},\tau)\in G_{n+m}$, such that $r^{-1}xr\in G_n\times G_m$, then there is $\eta_1\in \frakS_n$ and $\eta_2\in \frakS_m$ such that $\tau^{-1}\sigma\tau=i_1(\eta_1)i_2(\eta_2)$, but $i_1(\eta_1)i_2(\eta_2)$ is conjugated in $\frakS_{n+m}$ to $i_1(\eta_2)i_2(\eta_1)$, then there is $\gamma\in\frakS_{n+m}$ such that $\gamma^{-1}(\tau^{-1}\sigma\tau)\gamma=i_1(\eta_2)i_2(\eta_1)$. Then
$$((e,\ldots,e),\gamma^{-1})(r^{-1}xr)((e,\ldots,e),\gamma)\in G_m\times G_n.$$ It implies that we have bijective correspondence between elements $r\in G_{n+m}$ such that $r^{-1}xr\in G_n\times G_m$ and elements $s\in G_{n+m}$ such that $s^{-1}xs\in G_m\times G_n$. Then if we apply Theorem \ref{Main Theorem}, the number of terms in each direct sum computing $$\Ind_{G_n\times G_m}^{G_{n+m}}(E_1\boxtimes E_2)\text{ and }\Ind_{G_m\times G_n}^{G_{n+m}}(E_2\boxtimes E_2)$$ are the same. Moreover, for every $(\alpha,\beta)\in G_n\times G_m$, $$\charac_{G_n\times G_m}(\alpha,\beta)(E_1\boxtimes E_2)\cong\charac_{G_m\times G_n}(\beta,\alpha)(E_2\boxtimes E_1),$$then we have $E_1\cdot E_2$ and $E_2\cdot E_1$ have the same character, then the product is commutative. The other properties follows directly.

\end{proof}

\begin{definition}
	Let $R$ be a commutative ring, the graded-symmetric algebra of a $\IZ$-graded $R$-module $M$ (denoted by $\mathcal{S}(M)$) is the quotient of the tensor algebra of $M$ by the ideal $I$ generated by elements of the form
	\begin{enumerate}
		\item $x\otimes y -(-1)^{\deg(x)\deg(y)}(y\otimes x)$
		\item $x\otimes x$, when $\deg(x)$ is even.
	\end{enumerate}
\end{definition}
Now we will give another proof of the description of $\calf_G^q(X)$ as a graded-symmetric algebra given in \cite{segal1996} or Theorem 3 in \cite{wang2000}.
\begin{theorem}\label{symmetric}
	There is an isomorphism of $(\IZ_+\times\IZ/2\IZ)$-graded algebras $$\Phi:\calf_G^q(X) \to \mathcal{S}(\oplus_{n\geq1}q^nK_G^*(X)\otimes\IC).$$
\end{theorem}

\begin{proof}
	First note that using $\charac_{G_n}$ we can define an injective group homomorphism in the following way. Consider the following sequence of maps
	$$K_{G}^*(X)\otimes\IC\xrightarrow{\cong}\bigoplus_{g\in G_{*}}K^*(X^g)^{C_{G}(g)}\otimes\IC\xrightarrow{\lambda}\bigoplus_{x\in G_{n*}}K^*((X^{n})^x)^{C_{G_n}(x)}\otimes\IC\xrightarrow{\cong}K_{G_n}^*(X^n)\otimes\IC$$
	where the map $\lambda$ is given by the assigning  $[((g,1,\ldots,1),(1,\ldots,n))]_{G_n}$ to the conjugacy class $[g]_G$ and using the identification in Proposition \ref{cycle}. This map is certainly injective. Define $$\phi:K_{G}^*(X)\otimes\IC\to K_{G_n}^*(X^n)\otimes\IC$$ by the composition of the above sequence so that $\phi$ is injective and by the universal property of the graded-symmetric algebra we have a unique map $$\Phi:\mathcal{S}\left(\bigoplus_{n\geq1}K_G^*(X)\otimes\IC\right)\to \calf^q_G(X)$$extending $\phi$.
	
	Suppose inductively that $\im(\Phi)$ contains $K_{G_k}^*(X^k)\otimes\IC$ for $k<n$. Then by induction we know that the image of the following composition$$\mathcal{S}\left(\oplus_{n\geq1}q^nK_G^*(X)\otimes\IC\right)\times \mathcal{S}\left(\oplus_{n\geq1}q^nK_G^*(X)\otimes\IC\right)\xrightarrow{\Phi\times\Phi}\calf^q_{G}(X)\times \calf^q_{G}(X)\xrightarrow{\cdot}\calf^q_{G}(X)$$contains $$K_{G_k}^*(X^k)\otimes\IC\cdot K_{G_{n-k}}^*(X^{n-k})\otimes\IC\subseteq K_{G_n}^*(X)\otimes\IC.$$ Now we have that the image under  $\charac_{G_n}$ of
	
	$$\bigoplus_{k=1}^{n-1} ( K_{G_k}^*(X^k)\otimes\IC)\cdot( K_{G_{n-k}}^*(X^{n-k})\otimes\IC)$$ coincides with  $$\bigoplus_{x\in J}K^*(((X^n)^x)/C_{G_n}(x))\otimes\IC,$$where $J$ is the set of conjugacy classes in $G_n$  such that for every $c$, $m_\bullet(n, c)=0$, in other words $J$ is the set of conjugacy classes whose components in $\frakS_n$ are not an $n$-cycle.
	
	On the other hand if $x=((g,1,\ldots,1),(1,\ldots,n))$ for some $g\in G$, then  Proposition \ref{cycle} gives us 
	 that $\im(\charac_{G_n}\circ\Phi)$ contains $K^*((X^n)^x)^{C_{G_n}(x)}$. Finally, since $\charac_{G_n}$ is an isomorphism we can conclude that $\Phi$ is surjective.
	 
	 To see that $\Phi$ is injective we can use the formula for the induced character, because this formula implies that if $A\in \mathcal{S}(\bigoplus_{k\geq1}q^kK_G^*(X)\otimes\IC)$ is not zero then there exists $n$ and $x\in G_n$ such that $\charac_{G_n}(x)(\Phi(A))\neq0$, then $\Phi(A)\neq0$.

	\end{proof}
\section{Pullback of groups}\label{section5}
Let $\Gamma$ be a group fitting into the following pullback diagram 
\begin{equation}\label{pullbackrep}
\xymatrix{
	\Gamma \ar[r]^{p_2}\ar[d]_{p_1} & G\ar[d]_{\pi_2}\\
	H \ar[r]^{\pi_1} &  K}\end{equation}
If the group $\Gamma$ comes from a diagram \ref{pullbackrep} then it is isomorphic to a subgroup of $G\times H$, namely $\Gamma\cong\{(g,h)\in G\times H\mid \pi_1(g)=\pi_2(h)\}$.  When maps $\pi_1$ and $\pi_2$ are clear from the context we denote $\Gamma$ by $G\times_KH$. We suppose that $\pi_1$ and $\pi_2$ are surjective.

 In this section we describe the class function ring of $\Gamma$ in terms of the class function rings of $G$, $H$ and $K$. In order to obtain this description we need that $\Gamma$ satisfies the following condition.
 
 \begin{definition}\label{conjugacy-closed}
 	Let $G$ be a finite group and let $H\subseteq G$ be a subgroup, let $[h]_H\in H_*$, we say that $[h]_H$ is closed in $G$ if, $$[h]_H=[h]_G\cap H.$$  We say that $H$ is conjugacy-closed  in  $G$ if, for every $h\in H$, $[h]_H$ is closed in $G$. 
 \end{definition}

\begin{example}
	The following are examples of conjugacy-closed subgroups:
	\begin{itemize}
		\item The general linear groups over subfields are conjugacy-closed. 
		\item The symmetric group is conjugacy-closed in the general linear group. 
		\item The symmetric group on subsets are conjugacy-closed.
		\item The orthogonal group is conjugacy-closed in the general linear group over real numbers. 
		\item The unitary group is conjugacy-closed in the general linear group.
	\end{itemize}
\end{example}
\begin{remark} Let $G$ and $H$ be groups\\

\begin{itemize}
		
	\item If $H\subseteq G$ is conjugacy-closed in $G$, then the pullback of the inclusion$$i^*:\Class(G)\to\Class(H)$$ is surjective.

	\item If $H$ is a retract in $G$, the pullback of the inclusion$$i^*:R(G)\to R(H)$$is surjective.
	\end{itemize}
\end{remark}
 When $\Gamma$ is conjugacy-closed  in $G\times H$, we have a way to express the class function ring of $\Gamma$ in terms of the class function rings of $G$, $H$ and $K$. The same is true for the representations ring when $\Gamma$ is a retract of $G\times H$.
\subsection{The class function ring of a pullback}\label{section4}
Consider a pullback diagram of finite groups such as (\ref{pullbackrep}). If we apply the representation ring functor we obtain the following diagram
\begin{equation}\label{pullbackrepring}
\xymatrix{
	R(\Gamma)  & R(G)\ar[l]^{p_2^*}\\
	R(H)\ar[u]_{p_1^*}  &  R(K)\ar[l]^{\pi_1^*}\ar[u]_{\pi_2^*}}\end{equation}This diagram endows the rings $R(G)$ and $R(H)$ with a $R(K)$-module structure. A similar statement is true changing the representation ring by the class function ring. We will prove that if $\Gamma$ is a retract of $G\times H$, then the diagram (\ref{pullbackrepring}) is a pushout. In fact, we have the following theorem.

\begin{theorem}\label{pullbackrepringiso}
	Let $\Gamma$, $G$, $H$ and $K$ be finite groups such as in the diagram (\ref{pullbackrep}). If $\Gamma$ is conjugacy-closed  in $G\times H$, there is an isomorphism $$m:\Class(G)\otimes_{\Class(K)}\Class(H)\to \Class(\Gamma)$$of $\Class(K)$-modules
	
	Moreover, if $\Gamma$ is a retract of $G\times H$, we have an isomorphism
	$$f:R(G)\otimes_{R(K)}R(H)\to R(\Gamma)$$of $R(K)$-modules.
\end{theorem}
\begin{proof}
	In order to avoid confusion, in this proof we denote the product on $\Class(\Gamma)$, $\Class(G)$ and $\Class(H)$ by $\cdot$ and the generators of the tensor product by $\rho\otimes\gamma$.
	
	The map $f$ is defined as
	\begin{align*}
	f:\Class(G)\otimes_{\Class(K)}\Class(H)&\to \Class(\Gamma)\\
	\rho\otimes\gamma&\mapsto p_1^*(\rho)\cdot p_2^*(\gamma)\end{align*}

	First we prove that the map $f$ is well defined. Let $\xi\in \Class(K)$, $\rho\in \Class(G)$ and $\gamma\in \Class(H)$. 
	Let $(g,h)\in\Gamma$
	\begin{align*}
	f\left(\pi_1^*(\xi)\cdot\rho\otimes\gamma\right)(g,h)
	&=\left(p_1^*\left(\pi_1^*(\xi)\cdot\rho\right)\cdot p_2^*\left(\gamma\right)\right)(g,h)\\
	&=\left(\pi_1^*(\xi)\cdot\rho\right)(g)\gamma(h)\\
	&=\xi(\pi_1(g))\rho(g)\gamma(h)\\
	&=\rho(g)\xi(\pi_2(h))\gamma(h)\\
	&=f(\rho\otimes\pi_2^*(\xi)\cdot\gamma)(g,h).
	\end{align*}
	
	Now we will prove that $f$ is an isomorphism.
	Consider the following diagram with exact rows
	$$\xymatrix{
		0\ar[r]&\ker(\pi)
		 \ar[r]
		 \ar[d]^{f_2}& 
		 \Class(G)\otimes_\IC\Class(H)
         \ar[r]^{\pi}
         \ar[d]^{f_1}&\Class(G)\otimes_{\Class(K)}\Class(H)
         \ar[r]\ar[d]^{f}&0\\
		0\ar[r] &\ker(i^*)\ar[r]&\Class(G\times H)\ar[r]^{i^*}&\Class(\Gamma)\ar[r]&0. 
	}$$
	
	Where map $\pi$ is the quotient by the relations defining tensor product over $\Class(K)$, map $i^*$ is the pullback of the inclusion $i:\Gamma\to G\times H$, map $f_1$ is the natural isomorphism given by  tensor product over $\IC$ and the map $f_2$ is the restriction of $f_1$ to $\ker(\pi)$. Note that as $\Gamma$ is closed conjugacy in $G\times H$ the map $i^*$ is surjective. We will prove that above diagram is commutative and that $f_2$ is an isomorphism. 
	
	First we need to verify that $f_1(\ker(\pi))\subseteq \ker(i^*)$.
	Let $(g,h)\in \Gamma$,
	\begin{align*}
	i^* [f_1(\pi_1^*(\xi)&\cdot\rho\otimes\gamma-\rho\otimes\pi_2^*(\xi)\cdot\gamma)](g,h)= \\&[p_1^*(\pi_1^*(\xi))\cdot p_1^*(\rho)\cdot p_2^*(\gamma)-p_1^*(\rho)\cdot p_2^*(\pi_2^*(\xi))\cdot p_2^*(\gamma)](g,h)=0
	\end{align*}
	
	Now we prove that $\ker(i^*)= f_1(\ker(\pi))$. For this we will prove that if $f$ is a class function in $G\times H$ such that $i^*(f)\equiv 0$ and $f$ is orthogonal to every element in $f_1(\ker(\pi))$, then $f$ has to be zero.
	
	Suppose that for every $\xi\in \Class(K)$, $\rho\in \Class(G)$ and $\gamma \in \Class(H)$ we have
	$$\sum_{(g,h)\in G\times H}\overline{f(g,h)}\rho(g)\gamma(h)[\xi(\pi_2(h))-\xi(\pi_1(g))]=0.$$
	Let us fix $\rho\in \Class(G)$ and let
	$$\eta(g)=\sum_{h\in H} \overline{f(g,h)}\gamma(h)[\xi(\pi_2(h))-\xi(\pi_1(g))].$$
	We observe that $\eta$ is a class function on $G$ that is orthogonal to every $\rho$ in $\Class(G)$, then $\eta\equiv0$.
	
	By a similar argument we conclude that for every $(g,h)\in G\times H$ and $\xi\in \Class(K)$
	
	\begin{equation}\label{caracter}\overline{f(g,h)}[\xi(\pi_2(h))-\xi(\pi_1(g))]=0.\end{equation}
	
	We already know that $f(g,h)=0$ if $(g,h)\in\Gamma$, then let $(g,h)\notin \Gamma$, we have two cases. First suppose that $\pi_1(g)$ is conjugate to $\pi_2(h)$ in $K$, in this case there is $\bar{h}\in H$ such that $(g,\bar{h}h\bar{h}^{-1})\in\Gamma$ and then $f(g,h)=f(g,\bar{h}h\bar{h}^{-1})=0$.
	
	Suppose now that $\pi_1(g)$ is not conjugate to $\pi_2(h)$ in $K$, in this case there is  $\xi\in \Class(K)$ such that $\xi(\pi_1(g))\neq\xi(\pi_2(g))$ and equation \ref{caracter} gives us that $f(g,h)=0$. Then we conclude that $\ker(i^*)=f_1(\ker(\pi))$. The map $f_2$ is an isomorphism because it is the restriction of $f_1$ and as the diagram is commutative we conclude that $f$ is $\Class(K)$-module isomorphism.
	
	When $\Gamma$ is a retract in $G\times H$, the same argument works changing characters by representations, in particular the map $i^*:R(G\times H)\to R(\Gamma)$ is surjective.
\end{proof}	Observe that the pullback is not always conjugacy-closed in the product as the following example shows.
	\begin{example} Consider the pullback of the symmetric groups $\frakS_3$ over the cyclic group $C_2$
	\[\xymatrix{
		\Gamma \ar[r] \ar[d] & \frakS_3 \ar[d]^{sgn}\\
		\frakS_3 \ar[r]^{sgn} &  C_2
	}\]
	In this case the pullback $\Gamma$ has 6 conjugacy classes, $\{\gamma_1,\cdots, \gamma_6\}$. The product $\frakS_3\times \frakS_3$ has 9 conjugacy classes, $\{\chi_1\cdots,\chi_9\}$. Observe that the elements $((1,2,3), (1,2,3))$ and $((1,2,3),(1,3,2))$ are conjugate in the group $\frakS_3\times \frakS_3$ by the element $(e, (1,2))$ but they are not conjugate in $\Gamma$. 

The pullback of the inclusion can be described in the class function ring as follows:
\begin{eqnarray*}
	&\Class(\frakS_3\times \frakS_3)&\to \Class(\Gamma)\\
	&\chi_1 &\mapsto \gamma_1\\
	&\chi_2 &\mapsto \gamma_1\\
	&\chi_3 &\mapsto \gamma_2\\
	&\chi_4 &\mapsto \gamma_2\\ 
	&\chi_5 &\mapsto \gamma_3\\ 
	&\chi_6 &\mapsto \gamma_3\\ 
	&\chi_7 &\mapsto \gamma_4\\ 
	&\chi_8 &\mapsto \gamma_4\\
	&\chi_9 &\mapsto \gamma_5+\gamma_6.
\end{eqnarray*}
This map is not surjective.
\end{example}

\begin{example}Consider the following pullback
	
	\[\xymatrix{
		\Gamma \ar[r] \ar[d] & D_{12} \ar[d]_{\psi_1}\\
		C_3\ltimes C_4 \ar[r]^{\psi_2} &  \frakS_3
	}\]
	
	In this case the pullback $\Gamma$ is isomorphic to the group $C_2\times(C_3\ltimes C_4)$ and it is conjugacy closed in the group $D_{12}\times (C_3\ltimes C_4)$. According with GAP\cite{GAP4}  the group $D_{12}$ has group id (12,4) and generators $d_1, d_2$ and $d_3$ of orders $2,2$ and $3$ respectively. The homomorphism $\psi_1$ is given by 
	\begin{align*}
	\psi_1:D_{12}&\to \frakS_3\\
	d_1&\mapsto (2,3)\\
	d_2&\mapsto (1)\\
	d_3&\mapsto (1,2,3).
	\end{align*}
	
	The group $C_3\ltimes C_4$ has id (12,1) and it is generated by three elements $g_1, g_2, g_3$ of orders $4,2$ and $3$ respectively. The homomorphism $\psi_2$ is given by 
	\begin{align*}
	\psi_2 :C_3\ltimes C_4&\to \frakS_3\\
	g_1&\mapsto (2,3)\\
	g_2&\mapsto (1)\\
	g_3&\mapsto (1,2,3).
	\end{align*}
	
	For these groups the pullback $\Gamma$ is isomorphic to the group $C_2\times(C_3\ltimes C_4)$ with group id (24,7) and four generators $f_1, f_2, f_3, f_4$ of orders $4,6,2$ and $2$ respectively.

	Applying Theorem \ref{pullbackrepringiso} we obtain an isomorphism 
	$$\Class(\Gamma)\cong \Class(D_{12})\otimes_{\Class(\frakS_3)}\Class(C_3\ltimes C_4).$$
	
	For more examples please see \url{https://sites.google.com/site/combariza/research/pullbacks-with-kernel-s3}. 	
\end{example}
\section{Semidirect product of a direct product}
Let $A_1, A_2$ be groups with an action of a group $G$ by  automorphisms noted by $a_i^g= g\cdot a_i$, for $a_i\in A_i$ and $g\in G$. 
Note that $G$ acts also on the direct product $A_1\times A_2$ by acting on each component, i.e. $(a_1,a_2)^g := (a_1^g, a_2^g)$. In this section we describe the semidirect product of a direct product as a pullback of two semidirect products and then, we apply this for the wreath product of a direct product which will allow us to compute the Fock ring of a product. 

Consider the projections $\pi_i:A_i\rtimes G\to G$ and the pullback $\Gamma$ associated 
$$
\xymatrix{
	& \Gamma \ar[d] \ar[r] 
	& A_1\rtimes G \ar[d] \\
	& A_2\rtimes G \ar[r] 
	& G
}
$$
\begin{proposition} The pullback $\Gamma$ is isomorphic to the semidirect product $	(A_1\times A_2)\rtimes G.$
\end{proposition}
\begin{proof}
	Note that the pullback is the subgroup of $(A_1\rtimes G)\times(A_2\rtimes G)$ given by
	$$
	\Gamma =\{ (a_1,g_1,a_2,g_2) \in (A_1\rtimes G)\times(A_2\rtimes G) : \pi_1(a_1,g_1)=\pi_2(a_2,g_2), a_1\in A_1, g_i\in G \}
	$$
	that is, $g_1=g_2$. Consider the bijective function $\phi:\Gamma\to (A_1\times A_2)\rtimes G$ given by $\phi(a_1,g,a_2,g) = (a_1,a_2,g)$. On one hand 
	$$\phi[(a_1,g,a_2,g)\cdot(b_1,h,b_2,h)] = \phi(a_1b_1^g, gh, a_2b_2^g,gh) = (a_1b_1^g,a_2b_2^g,gh).$$ On the other hand $(a_1,a_2,g)\cdot(b_1,b_2,h) = (a_1b_1^g,a_2b_2^g,gh)$ which shows that $\phi$ is a homomorphism of groups. 
\end{proof}
\begin{corollary}
	Let $A,B$ be groups, there is an isomorphism 
	$$
	(A\times B)_n \cong A_n\times_{\frakS_n}B_n.
	$$
\end{corollary}
Now we proof that certain conjugacy classes in $(A\times B)_n$  are closed in $A_n\times B_n$.
\begin{proposition}Let $(\bar{g},\bar{h},\sigma)\in (A\times B)_n$, where $\sigma$ is an $n$-cycle. Then its conjugacy class in $A_n\times B_n$ is closed.
\end{proposition}
\begin{proof}
	Let $x=(\bar{g}_1,\bar{h}_1,\sigma_1)$ and $y=(\bar{g}_2,\bar{h}_2,\sigma_2)$ be elements in $(A\times B)_n$  that are conjugated in $A_n\times B_n$, where $\sigma_1$ and $\sigma_2$ are $n$-cycles. We can suppose that $\sigma_1=\sigma_2=(1\cdots n)$.
	
	Note that $$(\bar{g}_1,\sigma_1)\sim_{A_n}(\bar{g}_2,\sigma_2)\text{ and }(\bar{h}_1,\sigma_1)\sim_{B_n}(\bar{h}_2,\sigma_2).$$Since as $\sigma_1$ and $\sigma_2$ are $n$-cycles, $\prod_{i=1}^ng_{1,i}\sim_A\prod_{i=1}^ng_{2,i}$, and $\prod_{i=1}^nh_{1,i}\sim_B\prod_{i=1}^nh_{2,i}$. On the other hand the type of $x$ is given by
	$$
	m_x(r,c)=\begin{cases}1&\text{ if } r=n\text{ and }(\prod_{i=1}^ng_{1,i},\prod_{i=1}^nh_{1,i})\in c\\0&\text{ in any other case}
	\end{cases}$$
	and the type of $y$ is given by
	$$
	m_y(r,c)=\begin{cases}1&\text{ if } r=n\text{ and }(\prod_{i=1}^ng_{2,i},\prod_{i=1}^nh_{2,i})\in c\\0&\text{ in any other case}
	\end{cases}$$
	Then the types of $x$ and $y$ are equal, hence $x$ and $y$ are conjugated in $(A\times B)_n$.	\end{proof}
\section{The Fock space of a product of spaces}In this section we apply results of Section \ref{section5} in order to obtain a decomposition of $\calf_{G\times H}(X\times Y)$ in terms of $\calf_G(X)$ and $\calf_H(Y)$.
Let $X$ be a $G$-space, we can endow to $\calf_G(X)$ with natural module structures as follows:
\begin{itemize}
	\item Consider the trivial $G_n$-space $\pt$, and the unique $G_n$-map $\pi:X^n\to\pt$, then the pullback $$\pi^*:\Class(G_n)\to K_{G_n}^*(X^n)\otimes\IC$$ induces a $\Class(G_n)$-module structure over $K_{G_n}^*(X)\otimes\IC$, hence we have a $\calf_G(\pt)$-module structure over $\calf_G(X)$ defined componentwise.
	
	\item Note that we have a quotient map $s:G_n\to\frakS_{n}$, then the pullback $$(\pi\circ s)^*:\Class(\frakS_{n})\to K_{G_n}(X^n)\otimes\IC$$ induce a $\Class(\frakS_n)$-module structure over $K_{G_n}(X)\otimes\IC$, hence we have a $\calf(\pt)$-module structure over $\calf_G(X)$ defined componentwise.
\end{itemize}
As we observe in Section \ref{section4}, $(G\times H)_n$ is not closed conjugacy in $G_n\times H_n$, then we cannot expect a decomposition of $\Class((G\times H)_n)$ in terms of $\Class(G_n)$, $\Class(H_n)$ and $\Class(\frakS_n)$, but as the conjugacy classes with an $n$-cycle as component in $\frakS_n$ are closed we have a decomposition of $\calf_{G\times H}(X\times Y)$ in terms of $\calf_G(X)$, $\calf_H(Y)$ and $\calf(\pt)$, with $X$ a $G$-space and $Y$ a $H$-space.

\begin{theorem}There is an isomorphism of $\calf(\pt)$-modules
	$$
	\mathcal{F}_{G\times H}(X\times Y) \xrightarrow{\mathcal{K}} \mathcal{F}_G(X) \otimes_{\mathcal{F}(\pt)} \mathcal{F}_H(Y).
	$$The map $\mathcal{K}$ is compatible with the symmetric algebra decomposition in Thm. \ref{symmetric}. That means, we have a commutative diagram
		$$\xymatrix{\calf_{G\times H}(X\times Y)\ar[rr]^{K}\ar[d]^d&&\calf_G(X)\otimes_{\calf(\pt)}\calf_H(Y)\ar[d]^{d\otimes d}\\
		\mathcal{S}(X\times Y)\ar[rr]^{\pi_G\otimes\pi_H}&&\mathcal{S}(X)\otimes_{\calf(\pt)}\mathcal{S}(Y)}.$$
	Where $\mathcal{S}(X)$ stands for $\mathcal{S}(\bigoplus_{n\geq0}K_{G}(X)\otimes\IC)$, and similarly for $\mathcal{S}(Y)$ and $\mathcal{S}(X\times Y)$.
\end{theorem}
\begin{proof}
	Let $\{E_1,\ldots,E_m\}$ be a basis of $K_G^*(X)\otimes\IC$ as complex vector space and let $\{F_1,\ldots,F_s\}$ be a basis of $K_H^*(Y)\otimes\IC$ as complex vector space. From the proof of Theorem \ref{symmetric}  we can conclude that
	$$\{\Delta_{G,n,c,k}\in K_{G_n}(X^n)\otimes\IC\mid n\geq0, c\in G_*, 1\leq k\leq m\}$$ is a basis of $\calf_G(X)$ as $\IC$-algebra, where 
	$$\charac_{G_n}(\Delta_{G,n,c,k})((g_1,\ldots,g_n),\sigma)=\begin{cases}E_k&\text{if }\prod_{i=1}^n g_{\sigma^i(1)}\in c\text{ and }\sigma \text{ is an }n\text{-cycle}\\0&\text{in any other case}.\end{cases}$$In a similar way we define $$\{\Delta_{H,n,d,l}\in K_{H_n}(Y^n)\otimes\IC\mid n\geq0, d\in H_*, 1\leq l\leq s\} $$  a basis of $\calf_H(Y)$ as $\IC$-algebra, where
	$$\charac_{H_n}(\Delta_{H,n,d,l})((h_1,\ldots,h_n),\sigma)=\begin{cases}F_l&\text{if }\prod_{i=1}^n h_{\sigma^i(1)}\in d\text{ and }\sigma \text{ is an }n\text{-cycle}\\0&\text{in any other case.}\end{cases}$$
	Recall that we have an isomorphism $$K_{G\times H}(X\times Y)\otimes\IC\xrightarrow{\boxtimes} (K_G(X)\otimes K_H(Y))\otimes\IC,$$given by the external tensor product. It is proved in \cite{kunnethk-theory} or can be obtained (for complex coefficients) directly from the character. 
	
	Using the above identification we have that
	$$\{\Delta_{G\times H,n,c\times d,(k,l)}\mid n\geq0,c\in G_*, d\in H_*,1\leq k\leq m, 1\leq l\leq s\}$$is a basis as $\IC$-algebra of $\calf_{G\times H}(X\times Y)$, where 
	
	\begin{align*}\charac_{(G\times H)_n}(\Delta_{G\times H,n,c\times d,(k,l)})(\bar{g},\bar{h},\sigma)=\hspace{4.5cm}&\\\begin{cases}E_k\boxtimes F_l&\text{if }\prod_{i=1}^n g_{\sigma^i(1)}\in c, \prod_{i=1}^n h_{\sigma^i(1)}\in d\text{ and }\sigma \text{ is an }n\text{-cycle}\\0&\text{in any other case.}\end{cases}\end{align*}
	 As the conjugacy classes when the character of the above elements is not zero is closed in $G_n\times H_n$ we have $$\Delta_{G\times H,n,c\times d,(k,l)}=\Delta_{G,n,c,k}\cdot\Delta_{H,n,d,l},$$hence the map defined on generators as$$\Delta_{G\times H,n,c\times d,(k,l)}\mapsto\Delta_{G,n,c,k}\otimes\Delta_{H,n,d,l}$$is an isomorphism of $\calf(\pt)$-modules satisfying the required conditions.
	\end{proof}
	
\section{final remarks}
In \cite{ve2015} and \cite{ve2017} a configuration space representing equivariant connective K-homology for finite groups was constructed. We recall the construction briefly.

\begin{definition}\label{definconf}Let $G$ be a finite group and $(X,x_0)$ be a based $G$-connected, $G$-CW-complex.
	Let $\c(X, x_0,G)$ be the \textit{$G$-space of configurations of complex vector spaces over $(X,x_0)$}, defined as the increasing union, with respect to the inclusions $M_n(\IC[G])\rightarrow M_{n+1}(\IC[G])$
	$$\c(X,x_0,G)=\bigcup_{n\geq0}\Hom^*(C_0(X),M_n(\IC[G])),$$with the compact open topolo\-gy. 
	Notice that * refers to *-homomorphism, $C_0(X)$ denotes the C*-algebra of complex valued continuous maps vanishing at $x_0$ and $\IC[G]$ denotes the complex group ring.
	
	We endow  $\c(X,x_0,G)$ 
	with a continuous $G$-action as follows. If $F\in \c(X,x_0,G)$, we define 
	\begin{align*}g\cdot F: C_0(X)&\longrightarrow M_n(\IC[G])\\
	f&\longmapsto g\cdot F(g^{-1}\cdot f).\end{align*}\end{definition}

The space $\c(X,x_0,G)$ can be described as the configuration space whose elements are formal sums $$\sum_{i=1}^n(x_i,V_i),$$when $x_i\in X-\{x_0\}$ and $V_i\subseteq\IC[G]^\infty$ such that if $x_i\neq x_j$ then $V_i\perp V_j$, subject to some relations, for details see \cite[Sec. 2.1]{ve2017}. We call the elements $x_i$ the \emph{points} and to the vector spaces $V_i$ the \emph{labels}.

\begin{remark}
	When the based $G$-CW-complex $(X,x_0)$ is not supposed to be $G$-connected, we define the configuration space $$\c(X,x_0,G)=\Omega_0{}\c(\Sigma X,x_0,G),$$Where $\Omega_0$ denotes the based loop space and $\Sigma$ denotes the reduced suspension.
\end{remark}

That description allow us to define a Hopf space structure on $\c(X,x_0,G)$ by \emph{putting together} two configurations when labels in both of them are mutually orthogonal.

We have the following result:

\begin{theorem}[Thm. 5.2 in \cite{ve2015}]\label{conf}Let $(X,x_0)$ be a based finite $G$-connected $G$-CW-complex. If we denote by $k^G_n(X,x_0)$ the $n$-th $G$-equivariant  connective K-homology groups of the pair $(X,x_0)$, then there is a natural isomorphism 
	$$\pi_n(\c(X,x_0,G)^G)\xrightarrow{\mathfrak{A}^n}k_n^G(X,x_0).$$
\end{theorem}

When a Hopf space $\mathfrak{Y}$ is path-connected, consider the Hurewicz morphism 
$$\lambda:\pi_*(\mathfrak{Y};\IC)=\bigoplus_{n\geq0}\pi_i(\mathfrak{Y})\otimes\IC\to H_*(\mathfrak{Y};\IC).$$

We have the following result

\begin{theorem}[Thm. of the Appendix in \cite{mm65}]\label{mm}
	If $\mathfrak{Y}$ is a pathwise connected homotopy associative Hopf space with unit, and 
	$\lambda:\pi_*(\mathfrak{Y};\IC)\rightarrow H_*(\mathfrak{Y};\IC)$ is the Hurewicz morphism viewed as a morphism of $\IZ$-graded Lie algebras, then it induces an isomorphism of Hopf algebras $$\bar{\lambda}:\mathcal{S}(\pi(\mathfrak{Y};\IC))\rightarrow H_*(\mathfrak{Y};\IC).$$ 
\end{theorem} 
Applying the above theorem to $\c(X,x_0,G)$ we obtain.
\begin{corollary}\label{iso1}
Let $X$ be a finite $G$-CW-complex, if $X$ is $G$-connected we have an isomorphism 
$${\mathcal{S}}(k_*^G(X,x_0)\otimes\IC)\cong H_*(\c(X,x_0,G)^G;\IC).$$
\end{corollary}
In order to relate $H_*(\c(X,x_0,G)^G;\IC)$ with $\calf_G(X)$ we need to recall the following result proved in Theorem 6.13 in \cite{ve2015} using the equivariant Chern character obtained in \cite{lu2002}.

\begin{theorem}\label{chern} Let $X$ be a $G$-CW-complex. There is a natural isomorphism of $\IZ$-graded complex vector spaces (here the graduation is given by $q$)
	$$\bigoplus_{q\geq0}{k}_n^G(X)\otimes\IC\cong  \bigoplus_{n\geq0}{K}_n^G(X)\otimes\IC[q]. $$
\end{theorem}

Finally we can relate $H_*(\c(X,x_0,G)^G;\IC)$ with $\calf_G^q(X)$ when $X$ is an even dimensional $G$-connected, $G$-$\spinc$-manifold. First we recall \textit{Poincar\'e duality} for equivariant K-theory.

\begin{theorem}\cite{sch07}\label{poinca}
	Let $M$ be a $n$-dimensional $G$-$\spinc$-manifold. Then there exists an isomorphism
	$$D:K_G^*(M_+)\longrightarrow K^G_{n-*}(M).$$
\end{theorem}

Applying Theorem \ref{poinca} and Theorem \ref{mm} we can obtain the main result of the section.
\begin{theorem}\label{confmain}
	Let $(M,m_0)$ be an even dimensional $G$-connected, $G$-$\spinc$-manifold. We have an isomorphism of $\IZ$-graded Hopf algebras
	$$H_*(\c(M,m_0,G)^G;\IC)\cong\mathfrak{F}_G^q(M).$$
\end{theorem}
\begin{proof}
	Since $M$ is  a $G$-$\spinc$ manifold we can use Theorem 
	\ref{poinca}
	and obtain the following isomorphism of $\IZ_+\times\IZ/2\IZ$-graded Hopf algebras 
	\begin{align*}\s\left(k_*^G(M,m_0)\otimes\IC\right)&\cong {\s}\left(\bigoplus_{n\geq1}q^nK_*^G(M,m_0)\otimes\IC\right)\\&\cong {\s}\left(\bigoplus_{n\geq1}q^nK_G^*(M_+,+)\otimes\IC\right)\\&\cong 
	{\s}\left(\bigoplus_{n\geq1}q^nK_G^*(M)\otimes\IC\right).\end{align*}
	Combining Corollary \ref{iso1}, Theorem \ref{mm}  and Theorem \ref{symmetric} we obtain
	$$H_*(\c(M,m_0,G)^G;\IC)\cong\mathfrak{F}_G^q(M).$$\end{proof}

For the case when $M$ is not necessarily $G$-connected, we can obtain also a similar result. For details consult \cite[Proposition 6.11]{ve2015}.

\begin{proposition}\label{iso2}Let $X$ be a finite $G$-CW-complex, we have an isomorphism
	$$H_*(\Omega\c(\Sigma X,G)^G;\IC)\cong \s(k_*^G(X,x_0)\otimes\IC).$$
\end{proposition}

 In particular we have.
\begin{example}
	For $X=S^0$ we have $$\Omega\left(\c(\Sigma(S^0),G)\right)\simeq BU_G.$$Where $BU_G$ can be taken as the Grassmannian of finite dimensional vector subspaces of a complete $G$-universe. A \emph{complete $G$-universe} is a countably infinite-dimensional representation of $G$ with an inner product such that contains a copy of every irreducible representation of $G$, contains countably many copies of each finite-dimensional subrepresentation.  Applying the above discussion to this Hopf space we conclude that
	\begin{align*}
	H_*\left((BU_G)^G;\IC\right)\cong &R(G)\otimes\s\left(\pi_*((BU_G)^G)\otimes\IC\right)\\
	\cong& R(G)\otimes\s\left(\bigoplus_{n\geq0}R(G_n)\otimes\IC\right)\\
	\cong& \s\left(\bigoplus_{n\geq0}R(G_n)\otimes\IC\right).
	\end{align*}
	Summarizing, we have an isomorphism 
	
	$$H_*((BU_G)^G;\IC)\cong\calf_G^q(\pt)=\s\left(\bigoplus_{n\geq0}R(G_n)\otimes\IC\right).$$
\end{example}
We also have 
$$H_*\left((BU_G)^G;\IC\right)\cong\s\left(\bigoplus_{n\geq0}R(G_n)\otimes\IC\right)\cong\IC[\sigma_1^1,\ldots,\sigma_1^{k_1},\sigma_2^1,\ldots]$$
where $\{\sigma_i^1,\cdots,\sigma_i^{k_i}\}$ is a complete set of non isomorphic irreducible representations of $G_i$. We expect
that the elements $\sigma_i^k$ correspond in some sense with duals of $G$-equivariant Chern classes.

Now suppose that $M$ is a $G$-connected $G$-$\spinc$-manifold and $N$ is a $H$-connected $H$-$\spinc$-manifold, then we have an isomorphism of $\IZ$-graded Hopf algebras
$$H_*\left(\c(M\times N,(m_0,n_0),G\times H)^{G\times H};\IC\right)\cong \calf_{G}^q(M)\otimes_{\calf(\pt)} \calf_H^q(N)$$
In the case that $M=N=S^0$ with trivial action we obtain
$$H_*((BU_{G\times H})^{G\times H};\IC)\cong \calf_{G}(\pt)\otimes_{\calf(\pt)}\calf_H(\pt).$$
\bibliographystyle{alpha}
\bibliography{fock}
\end{document}